\documentclass[10pt,draft,reqno]{amsart}
     \makeatletter
     \def\section{\@startsection{section}{1}%
     \z@{.7\linespacing\@plus\linespacing}{.5\linespacing}%
     {\bfseries
     \centering
     }}
     \def\@secnumfont{\bfseries}
     \makeatother
\newtheorem{theorem}{Theorem}[section]
\newtheorem{lemma}[theorem]{Lemma}
\newtheorem{proposition}[theorem]{Proposition}
\newtheorem{corollary}[theorem]{Corollary}
\theoremstyle{definition}

\theoremstyle{remark}
\newtheorem{remark}[theorem]{Remark}
\numberwithin{equation}{section}
\begin{document}

\title[APPROXIMATIONS OF FRACTIONAL STOCHASTIC DIFFERENTIAL EQUATIONS]{APPROXIMATIONS OF FRACTIONAL STOCHASTIC DIFFERENTIAL EQUATIONS BY MEANS OF TRANSPORT PROCESSES*}

\author{Johanna Garz\'on}
\address{Johanna Garz\'on:  Department of Mathematics, CINVESTAV-IPN, Mexico City,
 Mexico}
\email{johanna@math.cinvestav.mx}

\author{Luis G. Gorostiza}
\thanks{* This research is supported by CONACYT grant 98998.}
\address{Luis G. Gorostiza: Department of Mathematics, CINVESTAV-IPN, Mexico City, Mexico }
\email{lgorosti@math.cinvestav.mx}

\author{Jorge A. Le\'on}
\address{Jorge A. Le\'on: Department of Automatic Control, CINVESTAV-IPN, Mexico City, Mexico }
\email{jleon@ctrl.cinvestav.mx}


\begin{abstract}
We present  strong approximations with  rate of convergence for the solution
of a stochastic differential equation of the form
$$
dX_t=b(X_t)dt+\sigma(X_t)dB^H_t,
$$
where $b\in C^1_b$, $\sigma \in C^2_b$, $B^H$ is fractional Brownian motion with Hurst index $H$, and we assume  existence of a unique solution with Doss-Sussmann representation. The results are based on a strong approximation of $B^H$ by means of transport processes of  Garz\'on et al (2009). If $\sigma$ is bounded away from $0$, an
 approximation is obtained by a general Lipschitz dependence result of R\"omisch and Wakolbinger (1985). Without that assumption on $\sigma$, that method does not work, and we proceed by means of Euler schemes on the Doss-Sussmann representation to obtain another approximation, whose proof is the bulk of the paper. 
\end{abstract}

\maketitle

\section{Introduction}

We consider one dimensional fractional 
 stochastic differential equations of the form
\begin{eqnarray}
\label{eq:1.1}
dX_t&=&b(X_t)dt+\sigma(X_t)dB^H_t,\quad t\in(0,T],\\
X_0&=&x_0,\nonumber
\end{eqnarray}
where $B^H$ is fractional Brownian motion with Hurst index $H$, and $b$ and $\sigma$ are continuous functions.
Equations of this type appear in several areas of application due to the properties of $B^H$ (see e.g. \cite{mis08} and references therein).
 $B^H=(B^H_t)_{t\geq 0}$ is defined for any $H\in (0,1)$ as a centered Gaussian process with covariance
$$
{\rm Cov}(B^H_s,B^H_t)=\frac{1}{2}(s^{2H}+t^{2H}-|s-t|^{2H}).
$$
We exclude the case $H=1/2$, which corresponds to  Brownian motion. The main properties of $B^H$ (for $H\neq 1/2$) are self-similarity, stationarity of increments, long range dependence, $k$-H\"older continuity of trajectories for $k<H$, and  it is neither a Markov process nor a semimartingale, hence the classical It\^o calculus cannot be used for this process. See \cite{nua03, sam94} for background.
There is now an extensive literature on fractional stochastic differential equations and applications. We restrict the references to a minimum for reasons of 
space.

In order to give a  precise meaning to the stochastic differential equation (\ref{eq:1.1}) it is necessary to define stochastic integral with respect to $B^H$ and to define solution, and then existence and uniqueness of solution can be proved under suitable conditions on $H,b$ and $\sigma$, in particular for the solution to have a 
Doss-Sussmann representation (given by equation (\ref{eq:2.9})). For example, this can be done for $H>1/3$, 
$b$ Lipschitz, and $\sigma\in C^2_b$. These questions have been treated  in \cite{nou08} (see also \cite{co07} and references therein). The paths of $B^H$ are increasingly irregular as $H$ decreases, and 
$H=1/4$ is a critical value for some problems (see e.g. 
\cite{cou00, fla09, gra03, kim08, nou09}). 
Hence the assumption $H>1/4$ is often made. In this case, under some conditions on $b$ and $\sigma$, a unique solution with Doss-Sussmann  representation has been established in \cite{alo01} for 
$1/4<H<1/2$ with Stratonovich integral, and in \cite{neu07} for $H>1/2$ with forward integral. In \cite{alo03} it is remarked that under some assumptions on the integrand, the Stratonovich integral coincides with the forward integral (see \cite{rus07} for those integrals). In this paper we take $b\in C^1_b$ and $\sigma \in C^2_b\;(C^i_b$ is the space of bounded functions with continuous bounded derivatives of order $\leq i)$, and we assume existence of a unique solution of (\ref{eq:1.1}) with Doss-Sussmann representation; this holds at least for $H>1/4$. The restriction $H>1/4$ is not used in our proofs, and it seems to be unnecessary if (\ref{eq:1.1}) is treated with the theory of rough paths; moreover, in that way it may also be possible to study strong transport approximations in the multidimensional case (see \cite{bar10, mar05, nua11, unt10} and references therein concerning the theory of rough paths and related fractional stochastic diferential equations).

In \cite{gar09} we obtained a strong approximation of $B^H$ with a rate of convergence by means of the Mandelbrot-van Ness representation of $B^H$ as a stochastic integral with respect to Brownian motion \cite{man68}, and a strong approximation of Brownian motion by transport processes \cite{goro80}. The result is that for each $q>0$ and each $\beta$ such that $|H-1/2|<\beta<1/2$, if $B^{n}$ is the $n$-th approximation of $B^H$ (defined in (\ref{eq:2.6})-(\ref{eq:2.7})), there is a positive constant $C$ such that 
\begin{equation}
\label{eq:1.2}
P\left(\sup_{0\leq t\leq T}|B^H_t-B^{n}_t|>C n^{-1/2+\beta}
(\log n)^{5/2}\right)=o(n^{-q})\quad{\rm as}\quad n\to\infty.
\end{equation}
Thus, the approximation becomes better when $H$ approaches 
${1}/{2}$. 

The aim of this paper is to show that under the above assumptions on $b$ and $\sigma,$ approximate solutions $X^n$ of equation (\ref{eq:1.1}) can be constructed by means of $B^n$, which strongly approximate the solution $X$ similarly as (\ref{eq:1.2}), with rate $n^{-1/2+\beta+\delta}(\log n)^{5/2}$, where 
$\delta>0$ is arbitrarily small, i.e., almost the same rate as (\ref{eq:1.2}). If $\sigma$ is bounded away from $0$, this can be done directly from (\ref{eq:1.2}) using a general Lipschitz dependence result obtained in \cite{rom85} (see also \cite{eng85}), without recourse to specifics of $B^n$. This is Theorem 3.1. Without that assumption on $\sigma$ the result of 
\cite{rom85} cannot be used (see Remark 4.1), and we proceed by  applying 
Euler schemes to the Doss-Sussmann representation of the solution. In this case, properties of $B^n$ are involved, and the approximation is somewhat different in the probability sense. This is  
Theorem 3.3, whose proof is the bulk of the paper.

We stress that the approximations of the solution of (\ref{eq:1.1}) by means of transport processes are of theoretical interest, but may not be useful for generating approximate solutions with computational efficiency. To that end there are  ad hoc methods (e.g. \cite{neu08} and references therein).

Section 2 contains background, Section 3 results, and Section 4 proofs.

\section{Background}

\subsection{Approximation of fractional Brownian motion by transport\\
 processes}

\

For each $n=1,2,\ldots$, a (uniform) transport process $Z^n=(Z^n(t))_{t\geq 0}$ represents the position on the real line at each time $t$ of a particle that starts from $0$ with velocity $+n$ or $-n$, with probability $1/2$ each, continues with that velocity during an exponentially distributed time with parameter $n^2$, at the end of which it changes velocity from $\pm n$ to $\mp n$, and so on the same way, changing  sign  at consecutive independent exponentially holding times. Such a process can be constructed from a given Brownian motion $B$ on a probability space using the Skorohod embedding, and it was shown in \cite{goro80} that $Z^n$ converges strongly to $B$ uniformly on a given bounded time interval  with rate $Cn^{-1/2}(\log n)^{5/2}$, as $n\to\infty$, where $C$ is a positive constant.

The Mandelbrot-van Ness representation of $B^H$ is given by 
\begin{equation}
\label{eq:2.1}
B^H_t=C_H\left(\int^0_{-\infty}f_t(s)dB(s)+\int^t_0g_t(s)dB(s)
\right),\quad 0\leq t\leq T,
\end{equation}
where
\begin{eqnarray}
\label{eq:2.2}
f_t(s)&=&(t-s)^{H-1/2}-(-s)^{H-1/2}
\quad{\rm for}\quad s<0\leq t,\\
\label{eq:2.3}
g_t(s)&=&(t-s)^{H-1/2}\quad{\rm for}\quad s<t,
\end{eqnarray}
$B$ is a Brownian motion on the whole real line, and $C_H$ is a positive constant \cite{man68}. Fix $a<0$. After an integration by parts and a change of variable, $B^H$ can be written as
\begin{eqnarray}
B^H_t&=&C_H
\left(
\int^t_0g_t(s)dB_1(s)+\int^0_af_t(s)dB_2(s)+f_t(a)B_2(a)\right.\nonumber\\ 
\label{eq:2.4}
&&\qquad -\int^0_{1/a}\partial_sf_t\left(\frac{1}{v}\right)
\left.\frac{1}{v^3}B_3(v)dv\right),
\end{eqnarray}
where $B_1, B_2$ and  $B_3$ are Brownian motions  given,  respectively, by the restriction of $B$ to $[0,T]$, the restriction of $B$ to $[a,0]$, and
$$B_3(s)=\left\{\begin{array}{lll}
s B(\frac{1}{s})&{\rm if}&s\in[{1}/{a},0),\\
0&{\rm if}& s=0.
\end{array}\right.
$$
To define the approximation $B^n=(B^n_t)_{t\geq 0}$ that appears in (\ref{eq:1.2}), 
the idea is to approximate $B_1,B_2$ and $B_3$ by corresponding transport processes $Z^n_1,Z^n_2$ and $Z^n_3$. 

For $n=1,2,\ldots$ and $ 0<\beta <1/2$, let
\begin{equation}
\label{eq:2.5}
\varepsilon_n=-n^{-\beta/|H-1/2|}.
\end{equation}
For $H>1/2$, define
\begin{eqnarray}
B^n_t&=&C_H\left(\int^t_0g_t(s)dZ^n_1(s)+\int^0_af_t(s)dZ^n_2(s)+f_t(a)Z^n_2(a)\right.
\nonumber\\
\label{eq:2.6}
&&\qquad +\left.\int^0_{1/a}
\biggl(-\int^{s\wedge\varepsilon_n}_{1/a}\partial_sf_t
\biggl(\frac{1}{v}\biggr)\frac{1}{v^3}dv\biggr)
dZ^n_3(s)\right),
\end{eqnarray}
and for $H<1/2$, define
\begin{align}
\label{eq:2.7}
B^n_t=&C_H\left(\int^{(t+\varepsilon_n)\vee 0}_0g_t(s)dZ^n_1(s)+\int^t_{(t+\varepsilon_n)\vee 0}g_t(\varepsilon_n+s)dZ^n_1(s)\right. \notag\\ 
& +\int^{\varepsilon_n}_af_t(s)dZ^n_2(s)+f_t(a)Z^n_2(a)+\int^0_{1/a}
\left(-\int^s_{1/a}\left.\partial_sf_t\left(\frac{1}{v}\right)
\frac{1}{v^3}dv
\right)dZ^n_3(s) \right).
\end{align}

Note that $B^n$ depends on $\beta$ through (\ref{eq:2.5}).

The following consequence of (\ref{eq:1.2}) is obvious since $q$ is arbitrary, hence large enough.
\begin{equation}
\label{eq:2.8}
P\left(\limsup_{ n\to\infty}
\left\{\sup_{0\leq t\leq T}|B^H_t-B^n_t|>Cn^{-
1/2+\beta}(\log n)^{5/2}\right\}\right)=0,
\end{equation}
where $\limsup$ is understood in the sense of sets. This will be used for the proof of Theorem 3.3.

\subsection{Doss-Sussmann representation}

\

Under suitable assumptions on $b$, $\sigma$ and $H$, the Doss-Sussmann representation     of (\ref{eq:1.1}) is given by (see \cite{alo01,dos77,neu07,nou08})
\begin{equation}                            
\label{eq:2.9}
X_t=h(Y_t,B^H_t),
\end{equation}
where the function $h$ and the process $Y$ are  the solutions of equations
\begin{equation}
\label{eq:2.10}
\frac{\partial h}{\partial x_2}(x_1,x_2)=\sigma(h(x_1,x_2)), \quad h(x_1,0)=x_1,\quad x_1,x_2\in\mathbb{R},
\end{equation}
and
\begin{equation}
\label{eq:2.11}
Y'_t={\rm exp}\left(-\int^{B^H_t}_0\sigma'(h(Y_t,s))ds\right)b(h(Y_t,B^H_t)),
\quad Y_0=x_0,
\end{equation}
respectively. The function $h$ has the property (see \cite{dos77}, Lemma 2)
\begin{equation}
\label{eq:2.12}
\frac{\partial h}{\partial x_1}(x_1,x_2)={\rm exp}\left(\int^{x_2}_0
\sigma'(h(x_1,s))ds\right),
\end{equation}
which implies that
\begin{equation}
\label{eq:2.13}
Y'_t=\left(\frac{\partial h}{\partial x_1}(Y_t,B^H_t)\right)^{-1}b(h(Y_t,B^H_t)), \quad Y_0=x_0.
\end{equation}

\section{Results}
Recall that $ b\in C^1_b$, $\sigma \in C^2_b$, (\ref{eq:2.9}) holds, and $B^n$ is defined by (\ref{eq:2.6})-(\ref{eq:2.7}).

The first theorem is a special case with the assumption that the function $\sigma$ is bounded away from $0$. The interest of this result is that  a direct proof can be given using Fernique's theorem \cite{fer70}, and a general Lipschitz dependence result of \cite{rom85}, without involving anything special about $B^n$.

Let $X^n$ be the solution of 
(\ref{eq:1.1}) with $B^H$ replaced by $B^n$, and the integral is defined pathwise.
\begin{theorem}
\label{t:3.1} 
Assume  $\sigma$ is bounded away from $0$. Let $|H-1/2|<\beta< 1/2$.
Then for each $\delta>0$ such that $\beta +\delta<1/2$, there exist $q>0$ and a positive constant $C$ such that
\end{theorem}
\begin{equation}
\label{eq:3.1}
P\left(\sup_{0\leq t\leq T}|X_t-X^n_t|>Cn^{-1/2+\beta+\delta}(\log n)^{5/2}\right)=o(n^{-q})\quad{\it as}\quad n\to \infty.
\end{equation}

\begin{remark}
\begin{enumerate}
	\item A transport approximation for equation (\ref{eq:1.1}) with Brownian motion $(H=1/2)$ and $\sigma$ bounded away from $0$ was studied in \cite{gor80,rom85}.
The result is like (\ref{eq:3.1}) with $\beta=\delta=0$. There was an error in the proof in \cite{gor80}, which was remedied with the method of \cite{rom85}. The paper \cite{cso88} gives a related result with a different type of formulation.
\item The approximation (\ref{eq:1.2}) holds for arbitrary $q$, but in the
approximation (\ref{eq:3.1}) $\delta$ is arbitrary and $q$ is chosen appropriately small (see the proof). Hence a result like (\ref{eq:3.2}) below may not hold under the assumption of Theorem 3.1.
\end{enumerate}
  \end{remark}

For the general result, let $X^n=h^n(Y^{n,n^2},B^n)$, where $(h^n)_n$  is an Euler scheme approximation of equation (\ref{eq:2.10}), and $(Y^{n,m})_m$ is an Euler scheme approximation of $Y^n$, which is the (Doss-Sussmann) solution of equation 
(\ref{eq:2.11}) with $B^H$ replaced by $B^n$ (the precise definitions of $h^n$ and $Y^{n,m}$ are given in (\ref{eq:4.3}) and
(\ref{eq:4.9})).

\begin{theorem}
\label{t:3.3}
Let $|H-1/2|<\beta < 1/2$.  Then for each $\delta>0$ such that $\delta <\beta$ and $\beta +\delta <1/2$, there exists a positive constant $C$ such that
\end{theorem}
\begin{equation}
\label{eq:3.2}
P\left(\limsup_{n\to\infty}\left\{\sup_{0\leq t\leq T}|X_t-X^n_t|>C
n^{-1/2+\beta+\delta}(\log n)^{5/2}\right\}\right)=0.
\end{equation}

Note that without the assumption on $\sigma$ in Theorem 3.1 the same rate of convergence holds, but the probabilistic part of the result is slightly different.

\section{Proofs}

We write $||\cdot||_\infty$ for the sup norm on [0,T].

\subsection{Proof of Theorem \ref{t:3.1}}
\

We give an idea of the proof.
By Fernique's theorem \cite{fer70} there exists $\alpha>0$ such that
$$P(||B^H||_\infty\geq s)\leq Ke^{-\alpha s^2},\quad s\geq 0,$$
where $K=E{\rm exp}(\alpha||B^H||^2_\infty)<\infty$. Then, applying Theorem 3 A) and Theorem 4 of \cite{rom85}, from (\ref{eq:1.2}) we obtain
$$P(||X-X^n||_\infty\geq n^\delta C 
n^{-{1}/{2}+\beta}(\log n)^{5/2})=o(n^{-q})\quad{\rm as}\quad n\to\infty,$$
with $\delta=q\gamma/\alpha^{1/2}$, where $\gamma$ is the positive constant in Theorem 4 of \cite{rom85}. Since $q$ in (\ref{eq:1.2}) is arbitrary, it can be chosen so that $\delta$ is as small as desired, and (\ref{eq:3.1}) is obtained. $\qed$

\begin{remark}

It can be shown that for $\sigma(x)=$ arctan $x$, formula (40) of \cite{rom85} cannot be proved because
$$\left(\frac{\partial h}{\partial x}(x,y)\right)^{-1}\geq {\rm exp}
\left(\frac{|y|}{1+x^2}\right)$$
for $x<0,y<0,$ so $(\frac{\partial h}{\partial x}(x,y))^{-1}$ is unbounded, and therefore the argument for the proof of Theorem 4 in \cite{rom85} is not valid.
\end{remark}

\subsection{Proof of Theorem \ref{t:3.3}}
\

Since the constant $C$ in (\ref{eq:2.8}) does not play an essential role, for simplicity we put $C=1$, and we will also prove the theorem with $C=1$ in 
(\ref{eq:3.2}).
  
We start by describing the parts of the approximation. We first approximate the function $h$ and the process $Y$ given by (\ref{eq:2.10}) and (\ref{eq:2.11}) respectively, and then, based on those approximations, we formulate the approximation for the solution of (\ref{eq:1.1}).

The function $h: \mathbb{R}^2 \to \mathbb{R}$ satisfies
\begin{equation}
\label{eq:4.1}
	h(x,y)=x+\int_0^y \sigma(h(x,s))ds.
\end{equation}
For each $n=1,2,\cdots$, we take the partition $\{y_i^n\}$ of the interval $[-n, n]$ given by $-n=y_{-n^2}^n < \cdots < y_{-1}^n < y_0^n = 0 < y_{1}^n< \cdots < y_{n^2}^n=n$, where for $r_n={1}/{n}$, and  $i=1, \cdots, n^2-1$,
\begin{equation}
\label{eq:4.2}
	y_{i+1}^n=y_i^n+r_n=\frac{i+1}{n},\quad
		y_{-(i+1)}^n=y_{-i}^n-r_n=-\frac{i+1}{n}.
\end{equation}
We define the functions $h^n: \mathbb{R}^2 \to \mathbb{R}$ by 
$$h^n(x,y)=0 \quad{\rm if}\quad (x,y)\notin [-n,n]\times[-n,n],$$
for $(x,y)\in [-n,n]\times[-n,n]$ and $k=0, 1, \cdots, n^2-1$, 
\begin{eqnarray*}
h^n(x,y^n_0)&=&x,\\
h^n(x,y^n_{k+1})&=& h^n(x, y^n_{k})+r_n\sigma(h^n(x, y^n_{k})),\\
	h^n(x,y^n_{-(k+1)})&=& h^n(x, y^n_{-k})-r_n\sigma(h^n(x, y^n_{-k})),
\end{eqnarray*}
and by linear interpolation,
\begin{align}
\label{eq:4.3}
	h^n(x,y)&= h^n(x, y^n_{k})+(y-y^n_{k})\sigma(h^n(x, y^n_{k}))\quad \ \ \ \ \ \  \text{if} \ \ \  y^n_{k}\leq y < y^n_{k+1},\notag \\
	h^n(x,y)&= h^n(x, y^n_{-k})- (y^n_{-k}-y)\sigma(h^n(x, y^n_{-k}))\ \ \ \  \text{if} \ \ \  y^n_{-(k+1)}<  y \leq y^n_{-k}.
\end{align}

From (\ref{eq:2.13}) we have 
\begin{equation}
\label{eq:4.4}
	Y_t= x_0 + \int_0^t\left(\frac{\partial{h}}{\partial x_1}(Y_s, B^H_s)\right)^{-1}b(h(Y_s, B^H_s))ds.
\end{equation}
For each $n=1,2, \cdots$, we define the process $Y^n$ as the solution of
\begin{equation}
	\label{eq:4.5}
	Y^n_t= x_0 + \int_0^t\left(\frac{\partial{h}}{\partial x_1}(Y^n_s, 
B^n_s)\right)^{-1}b(h(Y^n_s, B^n_s))ds, \ \ \ t\in[0, T].
\end{equation}
These processes exist and are unique because for each $n=1,2, \cdots$, the function 
$$(x, t)\mapsto \left(\frac{\partial{h}}{\partial x_1}(x, B^n_t)\right)^{-1}b(h(x, 
B^n_t)), \ \ t\in [0, T],$$
 satisfies  Lipschitz and linear growth conditions in $x$ (see (\ref{eq:2.12}), (\ref{eq:4.7}) and Corollary \ref{eq:4.5} below). Moreover, they are defined on the same probability space as the Brownian motion $B$ in (\ref{eq:2.1}), since the solution of (\ref{eq:4.5}) is given for each sample point, and $B^n$ is defined on the same probability space as $B$.

We give next  an Euler scheme for approximating the processes $Y^n$ for each $n=1,2, \cdots$.

By  (\ref{eq:2.12}),   equation (\ref{eq:4.5}) can be written as
\begin{equation}
\label{eq:4.6}
(Y_t^n)'=f(Y^n_t, B^n_t),\quad  Y_{0}^n=x_0, 	
\end{equation}
where $f(x,y)$ is defined by
\begin{equation}
	\label{eq:4.7}
	f(x,y)= \exp\left(-\int_0^y \sigma'(h(x,u))du\right)b(h(x,y)).
\end{equation}
  For each $n=1,2, \cdots$, we define $f^n(x,y)$ by
  \begin{equation}
	\label{eq:4.8}
	f^n(x,y)= \exp\left(-\int_0^y \sigma'(h^n(x,u))du\right)b(h^n(x,y)),
\end{equation}
with $h^n$  given by (\ref{eq:4.3}).

The Euler scheme $({Y}^{n,m})_m$ for equation 
(\ref{eq:4.6}) is defined as follows for each $m=1,2, \cdots$, the partition $0=t_0<\cdots <t_m=T$ of $[0, T]$ with $t_{i+1}=t_i+r_m$, and $r_m={T}/{m}$:
\begin{eqnarray}
{Y}^{n,m}_0&=&x_0, \nonumber\\
{Y}^{n,m}_{t_{k+1}}&=&{Y}^{n,m}_{t_{k}}+r_mf^n({Y}^{n,m}_{t_{k}}, B^n_{t_{k}}),
\quad k=0, \cdots, (m-1), \nonumber \\
{Y}^{n,m}_{t}&=&{Y}^{n,m}_{t_{k}}+(t-t_k)f^n({Y}^{n,m}_{t_{k}}, B^n_{t_{k}})
\nonumber \\
\label{eq:4.9}
&=&{Y}^{n,m}_{t_{k}}+ \int_{t_k}^t f^n({Y}^{n,m}_{t_{k}}, 
B^n_{t_{k}})ds, \quad{\rm if}\quad t_k\leq t < t_{k+1}.
\end{eqnarray}

Equation (\ref{eq:1.1})  has  a unique solution $X$ with  representation 
(\ref{eq:2.9}), and we define similarly  the approximation $X^n$  by means of the Euler schemes for $h$ and $Y^n$ as
\begin{equation}
\label{eq:4.10}
X^n_t=h^n({Y}^{n, n^2}_t, B^n_t),
\end{equation}
where $h^n$ and ${Y}^{n, n^2}$ are given by (\ref{eq:4.3}) and 
(\ref{eq:4.9}), respectively.

We then have
\begin{equation}
\label{eq:4.11}
|X_t-X^n_t|\leq H_1(t)+H_2(t)+H_3(t),
\end{equation}
where
\begin{eqnarray}
\label{eq:4.12}
H_1(t)&=&|h(Y_t,B^H_t)-h(Y^n_t,B^n_t)|,\\
\label{eq:4.13}
H_2(t)&=&|h(Y^n_t,B^n_t)-h(Y^{n,n^2}_t,B^n_t)|,\\
\label{eq:4.14}
H_3(t)&=&|h(Y^{n,n^2}_t,B^n_t)-h^n(Y^{n,n^2}_t,B^n_t)|.
\end{eqnarray}
The proof consists in obtaining estimates involving $H_1,H_2$ and $H_3$.

We will need the following 
 preliminary results. First a Lipschitz property of $B^n$.

\begin{lemma}
\label{l:4.2}
Let $B^n$ be defined by (\ref{eq:2.6})-(\ref{eq:2.7}). Then for all $n$ and for  $t_1, t_2\in[0, T]$,
$$	\left|B^n_{t_2}-B^n_{t_1}\right|\leq Kn^{1+\beta}\left|t_2-t_1\right|,
$$
where $K$ is a positive constant.
\end{lemma}
\begin{proof} First we take $H> 1/2$. It suffices to check the property for each one of the four functions on the r.h.s. of (\ref{eq:2.6}), using (\ref{eq:2.2}), 
(\ref{eq:2.3}) and the  transport processes. We will omit some of the calculations.

Note that $\partial_sf_t>0$ by a straighforward calculation (see 
\cite{gar09}, Lemma 3.1).

For $t_1<t_2$, by the mean value theorem,
\begin{align*}
&\left|\int_0^{t_2} g_{t_2}(s)dZ_1^{n}(s)-\int_0^{t_1} g_{t_1}(s)dZ_1^{n}(s)\right| \\
\leq &n\left(\left|\int_0^{t_1}\left[ (t_2-s)^{H-1/2}-(t_1-s)^{H-1/2}\right]ds\right|+\left|\int_{t_1}^{t_2} (t_2-s)^{H-1/2}ds\right|\right)\nonumber \\
=&\frac{n}{H+1/2}\left(t_2^{H+1/2} - t_1^{H+1/2}\right)\leq nT^{H-1/2}\left|t_2 - t_1\right|. 
\end{align*}
Again by the mean value theorem, 
\begin{align*}
&\left|\int_a^{0} f_{t_2}(s)dZ_2^{n}(s)-\int_a^{0} f_{t_1}(s)dZ_2^{n}
(s)\right| \\
\leq& n\int_a^{0}| (t_2-s)^{H-1/2}-(t_1-s)^{H-1/2}|ds \\	
=&\frac{n}{H+1/2}\bigl( [(t_2-a)^{H+1/2}-(t_1-a)^{H+1/2}]-[t_2^{H+1/2} - t_1^{H+1/2}] 	\bigr)r\\
\leq&\frac{n}{H+1/2}((t_2-a)^{H+1/2}-(t_1-a)^{H+1/2} )\leq n(T-a)^{H-1/2}\left|t_2 - t_1\right|. 
\end{align*}
Next, 
\begin{align*}
&|f_{t_2}(a)Z_2^{n}(a)-f_{t_1}(a)Z_2^{n}(a)|\leq n(-a)|(t_2-a)^{H-1/2}-(t_1-a)^{H-1/2}|\\
\leq& n(-a)^{H-1/2}(H-1/2)\left|t_2-t_1\right|.
\end{align*}
Finally, 
\begin{eqnarray*}
\lefteqn{\left|\int_{1/a}^{0}\biggl(-\int_{1/a}^{s\wedge {\varepsilon}_n}\partial_sf_{t_2}\biggl(\frac{1}{v}\biggr)\frac{1}{v^3}dv\biggr)dZ_3^{n}(s)\right.\left.-\int_{1/a}^{0}\biggl(-\int_{1/a}^{s\wedge {\varepsilon}_n}\partial_sf_{t_1}\biggl(\frac{1}{v}\biggr)\frac{1}{v^3}dv\biggr)dZ_3^{n}(s)
\right|\nonumber}\\
&\leq& n(H-1/2)\int_{1/a}^{0}\int_{1/a}^{s\wedge {\varepsilon}_n}|-(t_{2}-1/v)^{H-3/2}+(t_{1}-1/v)^{H-3/2}|(-1/v)^3dvds\nonumber\\
&\leq& n(H-1/2)(3/2-H)\left|t_2-t_1\right|\int_{1/a}^{0}\int_{1/a}^{s\wedge {\varepsilon}_n}(-v)^{-H-1/2}dvds,
\end{eqnarray*}
and, by (\ref{eq:2.5}),
\begin{align*}
\int_{1/a}^{0}\int_{1/a}^{s\wedge {\varepsilon}_n}(-v)^{-H-1/2}dvds\leq & \frac{1}{H-1/2}\int_{1/a}^{0}(-\varepsilon_n)^{-H+1/2}dv \notag \\
=&\frac{1}{H-1/2}n^{\beta}(-1/a),
\end{align*}
hence the result for $H>1/2$ follows.

We proceed similarly for $H< 1/2$. 
Let
$$
A_1=\left|\int_0^{(t_2+\varepsilon_n)\vee 0} g_{t_2}(s)dZ_1^{n}(s)-\int_0^{(t_1+\varepsilon_n)\vee 0} g_{t_1}(s)dZ_1^{n}(s)\right|.
$$
If $t_1+\varepsilon_n<0$ and $t_2+\varepsilon_n<0$, then $A_1=0$.
  If $t_1+\varepsilon_n<0$ and $t_2+\varepsilon_n>0$, then by the mean value theorem,
\begin{align*}
A_1&=\biggl|\displaystyle\int_0^{t_2+\varepsilon_n} (t_2-s)^{H-1/2}dZ_1^{n}(s)\biggr|\\
&\leq n\displaystyle\int_0^{t_2+\varepsilon_n} (t_2-s)^{H-1/2}ds\leq n(-\varepsilon_n)^{H-1/2}(t_2+\varepsilon_n)\leq n^{1+\beta}\left|t_2-t_1\right|,
\end{align*}
and if $t_1+\varepsilon_n>0$ and $t_2+\varepsilon_n>0$, again by the mean value theorem,
\begin{eqnarray*}
A_1&\leq& n\left[\int_0^{t_1+\varepsilon_n} [(t_1-s)^{H-1/2}-(t_2-s)^{H-1/2}]ds +\int_{t_1+\varepsilon_n}^{t_2+\varepsilon_n} (t_2-s)^{H-1/2}ds\right]\nonumber \\
&\leq& 2n(-\varepsilon_n)^{H-1/2}\left|t_2-t_1\right|=2n^{1+\beta}\left|t_2-t_1\right|.\nonumber
\end{eqnarray*}

Let
$$
A_2=\left|\int_{(t_2+\varepsilon_n)\vee 0}^{t_2} g_{t_2}(\varepsilon_n+s)dZ_1^{n}(s)-\int_{(t_1+\varepsilon_n)\vee 0}^{t_1} g_{t_1}(\varepsilon_n+s)dZ_1^{n}(s)\right|
$$
We show only one case. The others are similar.
If $t_1+\varepsilon_n<0$, $t_2+\varepsilon_n>0$ and $t_1\geq t_2+\varepsilon_n$,
\begin{eqnarray*}
A_2&=&\left|\int_{t_1}^{t_2}(t_2-\varepsilon_n-s)^{H-1/2}dZ_1^{n}(s)-\int_{0}^{t_2+\varepsilon_n} (t_1-\varepsilon_n-s)^{H-1/2}dZ_1^{n}(s)\right.\nonumber \\
&&\left. \hspace{0.5cm}+ \int_{t_2+\varepsilon_n}^{t_1} [(t_2-\varepsilon_n-s)^{H-1/2}-(t_1-\varepsilon_n-s)^{H-1/2}]dZ_1^{n}(s)\right|\nonumber \\
&\leq&\frac{2n}{H+1/2}[(t_2-\varepsilon_n-t_1)^{H+1/2}-(-\varepsilon_n)^{H+1/2}]\nonumber\\
 &\leq& 2n(-\varepsilon_n)^{H-1/2}\left|t_2-t_1\right|= 2n^{1+\beta}\left|t_2-t_1\right|.\nonumber 
\end{eqnarray*}

Next,
\begin{align*}
&\left|\int_a^{\varepsilon_n} f_{t_2}(s)dZ_2^{n}(s)-\int_a^{\varepsilon_n} f_{t_1}(s)dZ_2^{n}(s)\right|  \\
\leq&\frac{n}{H+1/2}\left[(t_2-\varepsilon_n)^{H+1/2}-(t_1-\varepsilon_n)^{H+1/2}\right]\nonumber \\
\leq& n(t_1-\varepsilon_n)^{H-1/2}\left|t_2-t_1\right|\leq n(-\varepsilon_n)^{H-1/2}\left|t_2-t_1\right|=n^{1+\beta}\left|t_2-t_1\right|,
\end{align*}
\begin{eqnarray*}
\left| f_{t_2}(a)Z_2^{n}(a)- f_{t_1}(a)Z_2^{n}(a)\right| &\leq& n (-a)(1/2-H)(t_1-a)^{H-3/2}\left|t_2-t_1\right|\nonumber \\
&\leq& n^{1+\beta} (-a)^{H-1/2}(1/2-H)\left|t_2-t_1\right|,\nonumber 
\end{eqnarray*}
\begin{align*}
&\left|\int_{1/a}^{0}-\biggl(\int_{1/a}^{s}\partial_sf_{t_1}\biggl(\frac{1}{v}\biggr)\frac{1}{v^3}dv\biggr)dZ_3^{n}(s) \right. \left.-\int_{1/a}^{0}\biggl(-\int_{1/a}^{s}\partial_sf_{t_2}\biggl(\frac{1}{v}\biggr)\frac{1}{v^3}dv\biggr)dZ_3^{n}(s) \right| \nonumber \\
\leq& n\int_{1/a}^{0}\int_{1/a}^{s}(1/2-H)\left[(t_{1}-1/v)^{H-3/2}-(t_{2}-1/v)^{H-3/2}\right](-1/v)^3dvds\nonumber\\ 
\leq& n\int_{1/a}^{0}\int_{1/a}^{s}(1/2-H)(3/2-H)(t_{1}-1/v)^{H-5/2}\left|t_2-t_1\right|(-1/v)^3dvds\nonumber\\ 
\leq& (3/2-H)(-1/a)^{3/2-H}n^{1+\beta}\left|t_2-t_1\right|. 
\end{align*}
\end{proof}
\begin{corollary}
\label{c:4.3}
For each $n=1,2, \cdots,$ $\left\|B^n\right\|_{\infty}<\infty$ a.s.
\end{corollary}
\begin{proof} 
immediate from Lemma \ref{eq:4.2}. 
\end{proof}

By the assumptions on $\sigma$ and $b$, we have the next bounds:
\begin{eqnarray}
\lefteqn{|b(x)|\leq M_1,\,\, |\sigma'(x)|\leq M_2, \,\,|\sigma''(x)|\leq M_3, \nonumber} \\
\label{eq:4.15}
&&\kern-.7cm |b(x)-b(y)|\leq M_4|x-y|, \,\, |b'(x)|\leq M_4 \quad{\rm and}\quad
 |\sigma(x)|\leq M_5, 
\end{eqnarray}
where $M_1, \cdots, M_5$ are constants. We put $\bar{M}=\max\{M_2, M_5\}$.

Next we present some properties of the function $h$.

\begin{lemma}
\label{l:4.4}
Let $h$ be defined by (\ref{eq:2.10}). Then
\begin{enumerate}
	\item \begin{equation}
\label{eq:4.16}
	\left|\frac{\partial{h}}{\partial x_1}(x, y)\right|\leq \exp(M_2\left|y\right|).
\end{equation}
\item \begin{equation}
\label{eq:4.17}
\left|\left(\frac{\partial{h}}{\partial x_1}(x, y)\right)^{-1}\right|=\left|\exp\left(-\int_0^{y}\sigma'(h(x, s))ds\right)\right| \leq \exp(M_2\left|y\right|).
\end{equation}
\item \begin{equation}
\label{eq:4.18}
\left|\frac{\partial}{\partial x_1}\left(\frac{\partial{h}}{\partial x_1}(x, y)\right)^{-1}\right| \leq M_3\left|y\right|\exp(2M_2\left|y\right|).
\end{equation}

\item \begin{equation}
\label{eq:4.19}
\left|h(x_1, y)-h(x_2, y)\right| \leq \exp(M_2\left|y\right|)\left|x_1-x_2\right|.
\end{equation}

\item \begin{equation}
\label{eq:4.20}
	\left|\left(\frac{\partial{h}}{\partial x_1}(x_1, y)\right)^{-1}-\left(\frac{\partial{h}}{\partial x_1}(x_2, y)\right)^{-1}\right|\leq M_3\left|y\right|\exp(2M_2\left|y\right|)\left|x_1-x_2\right|.
\end{equation}

\item \begin{equation}
\label{eq:4.21}
\left|b\left(h(x_1, y)\right)-b\left(h(x_2, y)\right)\right| \leq M_4\exp(M_2\left|y\right|)\left|x_1-x_2\right|.
\end{equation}

\item \begin{equation}
\label{eq:4.22}
\left|b\left(h(x, y_1)\right)-b\left(h(x, y_2)\right)\right| \leq M_4M_5\left|y_1-y_2\right|.
\end{equation}

\item \begin{align}\label{eq:4.23}
&\left|\left(\frac{\partial{h}}{\partial x_1}(x_1, y)\right)^{-1}b\left(h(x_1, y)\right)-\left(\frac{\partial{h}}{\partial x_1}(x_2, y)\right)^{-1}b\left(h(x_2, y)\right)\right|\notag\\
\leq&\exp(2M_2\left|y\right|)[M_1M_3 \left|y\right|+M_4]\left|x_1-x_2\right|.
\end{align}
\end{enumerate}
\end{lemma}

\begin{proof}
 $(1)$ and $(2)$ are immediate from (\ref{eq:2.12}) and (\ref{eq:4.15}).
\begin{enumerate}\addtocounter{enumi}{2}
	\item As $\sigma''$ is bounded by $M_3$, and by (\ref{eq:4.16}), 
(\ref{eq:4.17}), then
\begin{align*}
\left|\frac{\partial }{\partial x_1}
\left(\frac{\partial h}{\partial x_1}(x, y)\right)^{-1}\right|
&=\left|\exp\left(-\int_0^{y}\sigma'(h(x,u))du\right)\right|\left|\int_0^{y}\sigma''(h(x,u))\frac{\partial{h}}{\partial x_1}(x, u)du\right|\\
&\leq M_3|y|\exp(2M_2|y|).
\end{align*}

\item Follows by the mean value theorem and (\ref{eq:4.16}).

\item Follows  by the mean value theorem and (\ref{eq:4.18}).

\item As $b$ is Lipschitz, the result follows  by the mean value theorem and (\ref{eq:4.19}).

\item Using (\ref{eq:2.10}), $b$ being Lipschitz and $\sigma$  bounded, for some $\xi$ between $y_1$ and $y_2$ we have
\begin{align*}
\left|b\left(h(x, y_1)\right)-b\left(h(x, y_2)\right)\right|& \leq M_4\left|h(x, y_1)-h(x, y_2)\right|=M_4\left|\frac{\partial{h}}{\partial x_2}(x, \xi)\right||y_1-y_2|\\
&= M_4|\sigma(h(x, \xi))||y_1-y_2|\leq M_4M_5|y_1-y_2|.
\end{align*}
\item \begin{equation}
\label{eq:4.24}
\left|\left(\frac{\partial{h}}{\partial x_1}(x_1, y)\right)^{-1}b\left(h(x_1, y)\right)-\left(\frac{\partial{h}}{\partial x_1}(x_2, y)\right)^{-1}b\left(h(x_2, y)\right)\right|\leq A_1 + A_2,
\end{equation}
where 
$$
A_1=\left|\left(\frac{\partial{h}}{\partial x_1}(x_1, y)\right)^{-1}-\left(\frac{\partial{h}}{\partial x_1}(x_2, y)\right)^{-1}\right|\left|b(h(x_1, y))\right|,$$%
\vglue-.1cm
$$
A_2=\left|b(h(x_1, y))-b(h(x_2, y))\right|\left|\left(\frac{\partial{h}}{\partial x_1}(x_2, y)\right)^{-1}\right|.
$$
Using $b$ bounded by $M_1$, (\ref{eq:4.21}), (\ref{eq:4.17}), (\ref{eq:4.20}) and (\ref{eq:4.24}) we obtain the result. 
\end{enumerate}
\end{proof}

\begin{corollary}
\label{c:4.5}
Let $f$ be defined by (\ref{eq:4.7}). Then for each $n=1,2, \cdots$ the function
$$(x, t)\mapsto f(x, B^n_t)$$
 satisfies  Lipschitz and linear growth conditions in $x$.
\end{corollary}

\begin{proof}By (\ref{eq:2.12}),  Lemma \ref{l:4.2} and  parts 1 and 8 of Lemma 
\ref{l:4.4} , for each $x_1, x_2 \in \mathbb{R}$ and $t\in [0, T]$,
$$\left|f(x_1, B^n_t)-f(x_2, B^n_t)\right|\leq \exp(2M_2Kn^{1+\beta}T)[M_1M_3 Kn^{1+\beta}T+M_4]\left|x_1-x_2\right|,$$
and as $b$ is bounded by $M_1$,  
$$\left|f(x_1, B^n_t)\right|\leq M_1\exp(M_2Kn^{1+\beta}T).$$  
\end{proof}

\begin{lemma}
\label{l:4.6} 
Let $f$ be defined by (\ref{eq:4.7}). Then for each $n=1, 2, \cdots,$
\begin{enumerate}
	\item $\left|f(x, B^n_t)\right|\leq M_1\exp(M_2\left\|B^n\right\|_{\infty})$,
	\item $\left|f(x_1, B^n_t)-f(x_2, B^n_t)\right|\leq Z_1\left|x_1-x_2\right|$,
	\item $\left|f(x, B^n_{t_1})-f(x, B^n_{t_2})\right|\leq Z_2\left|
B_{t_1}^n-B_{t_2}^n\right|$,
\end{enumerate}
with the random variables
\begin{equation}
\label{eq:4.25}
Z_1=\exp(2M_2\left\|B^n\right\|_{\infty})[M_1M_3 \left\|B^n\right\|_{\infty}+M_4] 
\end{equation}
\begin{equation}
\label{eq:4.26}
Z_2=(M_1M_2+M_5M_4)\exp(M_2\left\|B^n\right\|_{\infty}).
\end{equation}
\end{lemma}
\begin{proof}\begin{enumerate}
\item Follows from (\ref{eq:2.12}), (\ref{eq:4.17}) and boundedness of $b$.
	\item Follows from (\ref{eq:4.23}).
\item By (\ref{eq:2.10}),
\begin{align*}
\frac{\partial f(x, B^n_t) }{\partial x_2}
=&-f(x,B^n_{t})\sigma'(h(x,B^n_t))\\
 &+ \exp\left(-\int_0^{B^n_t}\sigma'(h(x,u))du\right)b'(h(x,B^n_t))
	\sigma(h(x,B^n_t)).
\end{align*}
Since $\sigma, \sigma'$ and $b'$ are bounded, and by part 1  and 
(\ref{eq:4.17}), then
$$\left|\frac{\partial f(x, B^n_t) }{\partial x_2}\right|\leq (M_1M_2+M_5M_4)\exp(M_2\left\|B^n\right\|_{\infty}),$$
and we have the result by the mean value theorem. 
\end{enumerate}
\end{proof}

Now we do the approximation of $h$. 
For fixed $n$, we work in the  square $[-n, n]\times [-n, n]$. Let $l=n+m$ for some $m>0$, and consider the finer partition of  $[-n, n]$ given by $-n=y^l_{-nl}<\cdots <y^l_0=0<  \cdots <y^l_{nl}=n,$ as in (\ref{eq:4.2}) with 
$r_l={1}/{l}$.
\begin{lemma}
\label{l:4.7} 
Let $h$ and $h^l$ be given by (\ref{eq:4.1}) and (\ref{eq:4.3}), respectively. Then for  $(x,y)\in [-n,n]\times [-n, n]$ and  $l>n$, 
$$\left|h(x,y)-h^l(x,y)\right|\leq \bar{M}^2\frac{n}{l}\exp{(\bar{M}n)},$$
where $\bar{M}=\max\left\{M_2, M_5\right\}$ (see (\ref{eq:4.15})).
\end{lemma}

\begin{proof}

Assume $y>0$ (the case $y<0$ is proved similarly).

If $0=y_0^l <y \leq y_1^l$,
$$h(x,y)=x+\int_{y_0^l}^y \sigma(h(x,s))ds$$  
and 
 $$h^l(x,y)= h^l(x, y^l_{0})+ (y-y^l_0)\sigma(h^l(x, y^l_{0}))=x+\int_{y_0^l}^y \sigma(h^l(x,y_0^l))ds.$$
Then, using that $\sigma$ is a Lipschitz function,
\begin{align}
\label{eq:4.27}
|h(x,y)-h^l(x,y)| 
&\leq \int_{0}^y|\sigma(h(x,s))-\sigma(h^l(x,y_0^l))|ds\notag\\ 
&\leq \bar{M}\int_{0}^y|h(x,s)-h^l(x,s)|ds + \bar{M}\int_{0}^y|h^l(x,s)-h^l(x,y_0^l)|ds.
\end{align}
Since, for $y_0^l\leq s \leq y_1^l$,
\begin{eqnarray*}
|h^l(x,s)-h^l(x,y_0^l)|&\leq& |h^l(x,y_0^l)+ (s-y_0^l)\sigma( h^l(x,y_0^l))-h^l(x,y_0^l)|\\
&\leq&|s-y_0^l||\sigma( h^l(x,y_0^l))|\leq \bar{M}(y_1^l- y_0^l),
\end{eqnarray*}
then
\begin{equation}
\label{eq:4.28}
\bar{M}\int_{y_0^l}^y|h^l(x,s)-h^l(x,y_0^l)|ds\leq \bar{M}^2(y_1^l-y_0^l)^2= \bar{M}^2r_l^2.
\end{equation}
By (\ref{eq:4.27}), (\ref{eq:4.28}), and   Gronwall's lemma, for $0=y_0^l <y \leq y_1^l$,
 \begin{align}
\label{eq:4.29}
|h(x,y)-h^l(x,y)|&\leq \bar{M}^2r_l^2\exp(\bar{M}(y-y_0^l))
\leq \bar{M}^2r_l^2\exp(\bar{M}(y_1^l-y_0^l))\notag\\
&\leq\bar{M}^2r_l^2\exp(\bar{M}n).
\end{align}

We will prove by induction that for $k=0, \cdots, nl-1$, if  $y_{k}^l <y \leq y_{k+1}^l$, then
\begin{equation}
\label{eq:4.30}
|h(x,y)-h^l(x,y)|\leq\bar{M}^2r_l^2\left[\exp(\bar{M}(y_{k+1}^l-y_0^l))+\cdots+\exp(\bar{M}(y_{k+1}^l-y_k^l))\right].
\end{equation}
For $k=0$ we have the result by (\ref{eq:4.29}). 
For  $y_{k}^l <y \leq y_{k+1}^l$,
\begin{eqnarray*}
h(x,y)&=&x+\int_{y_0^l}^{y_k^l} \sigma(h(x,s))ds+\int_{y_k^l}^{y} \sigma(h(x,s))ds\\
&=&h(x,y_k^l)+\int_{y_k^l}^{y} \sigma(h(x,s))ds,
\end{eqnarray*}
and
$$h^l(x,y)=h^l(x,y_k^l)+(y-y_k^l)\sigma(h^l(x,y_k^l))=h^l(x,y_k^l)+\int_{y_k^l}^{y}\sigma(h^l(x,y_k^l))ds.$$
By induction on $k$, (\ref{eq:4.30}), and $\sigma$ being Lipschitz,
\begin{align}
\label{eq:4.31}
|h(x,y)-h^l&(x,y)|\leq |h(x,y_k^l)-h^l(x,y_k^l)|+\int_{y_k^l}^{y}\left|\sigma(h(x,s))-\sigma(h^l(x,y_k^l))\right|ds\notag\\
\leq& \bar{M}^2r_l^2\left[\exp(\bar{M}(y_{k}^l-y_0^l))+\cdots+\exp(\bar{M}(y_{k}^l-y_{k-1}^l))\right]\notag\\
+&\bar{M}\left(\int_{y_k^l}^{y}\left|h(x,s)-h^l(x,s)\right|ds+\int_{y_k^l}^{y}\left|h^l(x,s)-h^l(x,y_k^l)\right|ds\right).
\end{align}
Since, for $y_k^l\leq s \leq y_{k+1}^l$,
\begin{align*}
\left|h^l(x,s)-h^l(x,y_k^l)\right|&=|h^l(x,y_k^l)+ (s-y_k^l)\sigma( h^l(x,y_k^l))-h^l(x,y_k^l)|\\
&\leq|s-y_k^l||\sigma( h^l(x,y_k^l))|\leq \bar{M}(y_{k+1}^l- y_k^l),
\end{align*}
then
\begin{equation}
\label{eq:4.32}
	\int_{y_k^l}^{y}\left|h^l(x,s)-h^l(x,y_k^l)\right|ds\leq  \bar{M}(y_{k+1}^l- y_k^l)^2=\bar{M}r_l^2.
\end{equation}
By (\ref{eq:4.31}), (\ref{eq:4.32}), 
and  Gronwall's lemma,
\begin{align*}
&|h(x,y)-h^l(x,y)|\\
\leq&\bar{M}^2r_l^2\left[\exp(\bar{M}(y_{k}^l-y_0^l))+\cdots+\exp(\bar{M}(y_{k}^l-y_{k-1}^l))+1\right]\exp\left(\bar{M}(y_{k+1}^l-y_k^l)\right)\\ 
=&\bar{M}^2r_l^2\left[\exp(\bar{M}(y_{k+1}^l-y_0^l))+\cdots+\exp(\bar{M}(y_{k+1}^l-y_k^l))\right].
\end{align*}
Then for all $(x,y)\in [-n,n]\times [-n, n]$ and  $l>n$, there is some $k\in \{0, \cdots, nl-1\}$ such that $y_{k}^l <y \leq y_{k+1}^l$, and by 
(\ref{eq:4.30}),
\begin{eqnarray*}
|h(x,y)-h^l(x,y)|&\leq&\bar{M}^2r_l^2\left[\exp(\bar{M}(y_{k+1}^l-y_0^l))+\cdots+\exp(\bar{M}(y_{k+1}^l-y_k^l))\right] \\
&\leq& \bar{M}^2r_l^2(k+1)\exp(\bar{M}n)
\leq  \bar{M}^2r_l^2nl\exp(\bar{M}n) \\
&=&\bar{M}^2\frac{n}{l}\exp(\bar{M}n). 
\end{eqnarray*}
\end{proof}

Next we do the approximation of $Y$.

We denote 
\begin{equation}
\label{eq:4.33}
\alpha_n=n^{-1/2+\beta+\delta}(\log n)^{5/2},
\end{equation}
with $\beta, \delta$ such that $\left|H-1/2\right|<\beta<1/2$,  $0< \delta < \beta$ and $\beta + \delta <1/2$.

\begin{proposition}
\label{p:4.8}
Let $Y$ and $Y^n$ be the processes given by (\ref{eq:4.4}) and (\ref{eq:4.5}), res-pectively. Then
$$P\left(\limsup_{n\to \infty}\left\{\left\|Y-Y^n\right\|_{\infty}>\alpha_n\right\}\right)=0,$$
where $\alpha_n$ is defined by (\ref{eq:4.33}).
\end{proposition}

\begin{proof}
\begin{eqnarray}
\label{eq:4.34}
|Y_t-Y_t^n|&=&\left|\int_0^t\left(\frac{\partial{h}}{\partial x_1}(Y_s, 
B^H_s)\right)^{-1}b(h(Y_s, B^H_s))ds\right.\nonumber\\ 
&&\hspace{3.5cm} \left.-\int_0^t\left(\frac{\partial{h}}{\partial x_1}(Y^n_s, B^n_s)\right)^{-1}b(h(Y^n_s, B^n_s))ds\right|\nonumber\\ 
&\leq& \int_0^t I_1(s)ds+\int_0^t I_2(s)ds,
\end{eqnarray}
where
\begin{equation}
\label{eq:4.35}
I_1(s)=\left|\left(\frac{\partial{h}}{\partial x_1}(Y_s, B^H_s)\right)^{-1}-\left(\frac{\partial{h}}{\partial x_1}(Y^n_s, B^n_s)\right)^{-1}\right|\left|b(h(Y^n_s, B^n_s))\right|,
\end{equation}
and
\begin{equation}
\label{eq:4.36}
I_2(s)=\left|\left(\frac{\partial{h}}{\partial x_1}(Y_s, B^H_s)\right)^{-1}\right|\left|b(h(Y_s, B^H_s))-b(h(Y^n_s, B^n_s))\right|.
\end{equation}
As $b$ is bounded by $M_1$, then
\begin{equation}
\label{eq:4.37}
\int_0^t	I_1(s)ds\leq M_1\int_0^t\left|\left(\frac{\partial{h}}{\partial x_1}(Y_s, B^H_s)\right)^{-1}-\left(\frac{\partial{h}}{\partial x_1}(Y^n_s, B^n_s)\right)^{-1}\right|ds.
\end{equation}
Denoting 
$$F(x_1,x_2)=\left(\frac{\partial{h}}{\partial x_1}(x_1, x_2)\right)^{-1}=\exp\left(-\int_0^{x_2}\sigma'(h(x_1,u))du\right),$$
(see (\ref{eq:2.12})), we have
\begin{equation}
\label{eq:4.38}
\left|\left(\frac{\partial{h}}{\partial x_1}(Y_s, B^H_s)\right)^{-1}-\left(\frac{\partial{h}}{\partial x_1}(Y^n_s, B^n_s)\right)^{-1}\right|\leq I_3(s)+I_4(s),
\end{equation}
where 
$$I_3(s)= \left|F(Y_s, B^H_s)-F(Y^n_s, B^H_s)\right|,$$
\vglue-.3cm 
$$I_4(s)=\left|F(Y^n_s, B^H_s)-F(Y^n_s, B^n_s)\right|.$$
By  (\ref{eq:4.20}),  
\begin{equation}
\label{eq:4.39}
I_3(s)	\leq M_3||B^H||_{\infty}\exp(2M_2||B^H||_{\infty})|Y_s-Y_s^n|.
\end{equation}
Since $\sigma'$ is bounded by $M_2$, and by (\ref{eq:4.17}),
$$
\left|\frac{\partial F}{\partial x_2}(x_1, x_2)\right|=\left|\exp\left(-\int_0^{x_2}\sigma'(h(x_1,u))du\right)\right||\sigma'(h(x_1,x_2))|
\leq M_2\exp(M_2|x_2|).
$$

By (\ref{eq:2.8}) there is measurable subset $A$ of the underlying sample space with $P(A)=1$, and for each  $\omega \in A$  there is a positive integer $\hat{N}=\hat{N}(\omega)$ such that
\begin{equation}
\label{eq:4.40}
||B^H(\omega)-B^n(\omega)||_{\infty}<1 \ \   \text{for all} \ \ n>\hat{N}. 
\end{equation}
For each  $\omega \in A$ fixed and $n>\hat{N}(\omega)$,  by the mean value theorem,  for some $\bar{r}(\omega)$ between $B^H_s(\omega)$ and $B_s^n(\omega)$, $|\bar{r}(\omega)|\leq 1+||B^H(\omega)||_{\infty}$, thus (omitting $\omega$)
 \begin{eqnarray}
\label{eq:4.41}
I_4(s)	&=&\left|\frac{\partial F}{\partial x_2}(Y_s^n, \bar{r})\right||B^H_s-
B_s^n|\leq M_2\exp(M_2|\bar{r}|)|B^H_s-B_s^n|\nonumber\\
&\leq& M_2\exp(M_2(1+||B^H||_{\infty}))||B^H-B^n||_{\infty}\quad
{\rm for} \quad n>\hat{N}.
\end{eqnarray}
From (\ref{eq:4.37}), (\ref{eq:4.38}), (\ref{eq:4.39}) and 
(\ref{eq:4.41}), for $n>\hat{N}$,
\begin{eqnarray}
\label{eq:4.42}
\lefteqn{\int_0^t I_1(s)ds \leq M_1\int_0^t[M_3||B^H||_{\infty}\exp(2M_2
||B^H||_{\infty})|Y_s-Y_s^n|\nonumber}\\  
&&\hspace{1.5cm} +M_2\exp(M_2(1+||B^H||_{\infty}))|B^H_s-B_s^n|]ds.
\end{eqnarray}
Now, by (\ref{eq:4.17}) and (\ref{eq:4.36}),
\begin{equation}
\label{eq:4.43}
I_2(s)\leq \exp(M_2||B^H||_{\infty})(I_5(s)+I_6(s)),
\end{equation}
where
$$I_5(s)=\left|b(h(Y_s, B^H_s))-b(h(Y^n_s, B^H_s))\right|,$$
\vskip-.3cm
$$I_6(s)=\left|b(h(Y^n_s, B^H_s))-b(h(Y^n_s, B^n_s))\right|.$$
From (\ref{eq:4.21}) and (\ref{eq:4.22}),
\begin{equation}
\label{eq:4.44}
I_5(s)\leq M_4\exp(M_2||B^H||_{\infty})|Y_s-Y_s^n|,
\end{equation}
\vskip-.3cm
\begin{equation}
\label{eq:4.45}
I_6(s)\leq M_4M_5|B^H_s-B_s^n|.
\end{equation}
By (\ref{eq:4.43}), (\ref{eq:4.44}) and (\ref{eq:4.45}),
\begin{align}
\label{eq:4.46}
\int_0^t I_2(s)ds \leq \int_0^t&\exp(M_2||B^H||_{\infty})\left[M_4\exp(M_2
||B^H||_{\infty})|Y_s-Y_s^n|\right.\notag\\
&\left. +M_4M_5|B^H_s-B_s^n|\right]ds.
\end{align}   
Therefore, for $\omega\in A$ there is $\hat{N}=\hat{N}(\omega)$, as above, such that for $n>\hat{N}$, by (\ref{eq:4.34}), (\ref{eq:4.42}) and (\ref{eq:4.46}),
$$
|Y_t-Y_t^n|\leq \alpha ||B^H-B^n||_{\infty} + \kappa \int_0^t|Y_s-Y_s^n|ds,
$$
where $\alpha$ and $\kappa$ are the random variables
$$\alpha=(TM_1M_2\exp(M_2)+TM_4M_5)\exp(M_2
||B^H||_{\infty}),$$ 
\vskip-.3cm
$$\kappa=(M_1M_3||B^H||_{\infty}+M_4)\exp(2M_2
||B^H||_{\infty}).$$
By Gronwall's lemma
$$|Y_t-Y_t^n|\leq Z||B^H-B^n||_{\infty},$$
where $Z$ is the random variable
$$
	Z=\alpha\exp(T\kappa).
$$
Then, for $n>\hat{N}$,
$$\left\|Y-Y^n\right\|_{\infty}\leq Z
||B^H-B^n||_{\infty},$$
and it follows that on $A$,
$$
\limsup_{n\to \infty}\left\{\left\|Y-Y^n\right\|_{\infty}>\alpha_n\right\}\subseteq \limsup_{n\to \infty}
\left\{Z||B^H-B^n||_{\infty}>\alpha_n\right\}.$$
Hence
\begin{equation}
\label{eq:4.47}
P\left(\limsup_{n\to \infty}\left\{\left\|Y-Y^n\right\|_{\infty}>\alpha_n\right\}\right)\leq P\left(\limsup_{n\to \infty}\left\{Z
||B^H-B^n||_{\infty}>\alpha_n\right\}\right).	
\end{equation}

Next we show that 
$$
\limsup_{n\to \infty}
\left\{Z||B^H-B^n||_{\infty}>\alpha_n\right\}
\subseteq \limsup_{n\to \infty}\left\{
||B^H-B^n||_{\infty}>
n^{-1/2+\beta}(\log n)^{5/2}\right\}.
$$
Since 
\begin{eqnarray*}
&&\limsup_{n\to \infty}\left\{Z
||B^H-B^n||_{\infty}>
n^{-1/2+\beta+ \delta}(\log n)^{5/2}\right\}\\
&=&\limsup_{n\to \infty}\left\{(Zn^{-\delta})
||B^H-B^n||_{\infty}>n^{-1/2+\beta}(\log n)^{5/2}\right\},
\end{eqnarray*}
then, by (\ref{eq:2.8}) we have
\begin{eqnarray}
\label{eq:4.48}
&&P\left(\limsup_{n\to \infty}\left\{Z
||B^H-B^n||_{\infty}>\alpha_n\right\}\right)\nonumber \\
&\leq & P\left(\limsup_{n\to \infty}\left\{
||B^H-B^n||_{\infty}>
n^{-1/2+\beta}(\log n)^{5/2}\right\}\right)=0.
\end{eqnarray}
The proof is finished by (\ref{eq:4.47}) and (\ref{eq:4.48}). 
\end{proof}

Now we do the approximation of $Y^n$.

\begin{lemma}
\label{l:4.9}
 For every $n,m $ and ${Y}^{n,m}$ defined by (\ref{eq:4.9}),
\begin{equation}
\label{eq:4.49}
|{Y}^{n,m}_t|\leq |x_0|+TM_1\exp(M_2\left\|B^n\right\|_{\infty}), \ \ \ t\in[0, T].
\end{equation}
\end{lemma}

\begin{proof}
 As $b$ and $\sigma$ are bounded, and by (\ref{eq:4.17}),   
then, for $f^n$ defined by (\ref{eq:4.8}),

\begin{equation}
\label{eq:4.50}
\left|f^n({Y}_t^{n,m}, B^n_t)\right|\leq M_1\exp(M_2\left\|B^n\right\|_{\infty}).	
\end{equation}

We will first prove that for all $k=0,\cdots, m-1$, 
\begin{equation}
\label{eq:4.51}
|{Y}^{n,m}_{t_k}|\leq |x_0|+kr_mM_1\exp(M_2\left\|B^n\right\|_{\infty}).
\end{equation}
For $k=0$ it is obvious. If we assume that
$$
|{Y}^{n,m}_{t_{k-1}}|\leq |x_0|+(k-1)r_mM_1\exp(M_2\left\|B^n\right\|_{\infty}),$$
then from (\ref{eq:4.50}) and (\ref{eq:4.9}),
\begin{eqnarray*}
|{Y}^{n,m}_{t_k}|&=&|{Y}^{n,m}_{t_{k-1}} + r_mf^n({Y}^{n,m}_{t_{k-1}}, B^n_{t_{k-1}})|\\
&\leq& |x_0|+(k-1)r_mM_1\exp(M_2\left\|B^n\right\|_{\infty})+ r_mM_1\exp(M_2\left\|B^n\right\|_{\infty})\\
&=&|x_0|+kr_mM_1\exp(M_2\left\|B^n\right\|_{\infty}).	
\end{eqnarray*}
Now, if $t_k\leq t < t_{k+1}$, by (\ref{eq:4.50}) and (\ref{eq:4.51}),
\begin{eqnarray*}
|{Y}^{n,m}_t|&=&|{Y}^{n,m}_{t_k} + (t-t_k)f^n({Y}^{n,m}_{t_k}, B^n_{t_{k}})|\\
&\leq& |x_0|+kr_mM_1\exp(M_2\left\|B^n\right\|_{\infty})+ r_mM_1\exp(M_2\left\|B^n\right\|_{\infty})\\
&=&|x_0|+(k+1)r_mM_1\exp(M_2\left\|B^n\right\|_{\infty})	\\
&\leq& |x_0|+mr_mM_1\exp(M_2\left\|B^n\right\|_{\infty}),	
\end{eqnarray*}
and (\ref{eq:4.49}) is obtained.	
\end{proof}

\begin{lemma}
\label{l:4.10}
There exists $N=N(\omega)>0$ such that for every $n, m $ and $t\in [0,T]$,
\begin{equation}
\label{eq:4.52}
({Y}^{n,m}_t, B^n_t)\in [-N, N]\times [-N, N] \ \ \ a.s.
\end{equation}
\end{lemma}

\begin{proof}
Let $\omega \in A$ and $\hat{N}=\hat{N}(\omega)$ 
(see (\ref{eq:4.40})),  then for all $n>\hat{N}$, 
\begin{equation}
\label{eq:4.53}
\|B^n\|_{\infty}\leq 1 + \|B^H\|_{\infty}
\end{equation}
and by the Lemma \ref{eq:4.2}, for $n\leq \hat{N}$,
\begin{equation}
\label{eq:4.54}
\|B^n\|_{\infty}\leq K\hat{N}^{1+\beta}T.
\end{equation}
Hence, by (\ref{eq:4.49}),
\begin{equation}
\label{eq:4.55}
|{Y}^{n,m}_t|\leq
\begin{cases}
|x_0|+TM_1\exp(M_2(1+
||B^H||_{\infty})) & \text{if} \ \ n> \hat{N},\\
	|x_0|+TM_1\exp(M_2 K\hat{N}^{1+\beta}T) & \text{if} \ \ n\leq \hat{N}.
\end{cases}
\end{equation}
Taking
\begin{equation}
\label{eq:4.56}
N=N(\omega)
=\max\begin{Bmatrix}
 1+||B^H||_{\infty}, K\hat{N}^{1+\beta}T,  |x_0|+TM_1\exp(M_2(1+
||B^H||_{\infty})),\\ |x_0|+TM_1\exp(M_2 K\hat{N}^{1+\beta}T) 
\end{Bmatrix},
\end{equation}
we then have the result by 
(\ref{eq:4.53}), (\ref{eq:4.54}), (\ref{eq:4.55}) and 
(\ref{eq:4.56}). 
 \end{proof}
 
\begin{corollary}
\label{c:4.11}
Let $N$ be given by Lemma \ref{l:4.10}. Then for all $n>{N}$ and $t\in [0,T]$,
$$
|h({Y}^{n,m}_t,B^n_t)-h^n({Y}^{n,m}_t,B^n_t)|\leq \bar{M}^2\frac{N}{n}\exp{(\bar{M}N)},$$
where $\bar{M}=\max\left\{M_2, M_5\right\}$.
\end{corollary}

\begin{proof}
 Follows from  Lemma \ref{l:4.7} and (\ref{eq:4.52}).
  \end{proof}

\begin{lemma}
\label{l:4.12}  
For every $n>N$, $m $ and $t\in[0,T]$,
\begin{equation}
\label{eq:4.57}
|Y^n_{t}-{Y}^{n,m}_t|\leq J_{n,m}T\left(\exp\left(Z_1T\right)\right)\ \ \ \text{a.s.},
\end{equation}
where $J_{n,m}$ is the random variable
\begin{align}
\label{eq:4.58}
J_{n,m}=&Z_1M_1\exp(M_2\left\|B^n\right\|_{\infty})r_m+Z_2Kn^{1+\beta}r_m\notag \\
&+ \bar{M}^2\exp(M_2\|B^n\|_{\infty})\exp{(\bar{M}N)}(M_1M_3\|B^n\|_{\infty}+M_4)\frac{N}{n},
\end{align}
and $Z_1, Z_2$ and $ N$ are the random variables  defined by 
(\ref{eq:4.25}),   (\ref{eq:4.26}) and  (\ref{eq:4.56}) respectively, and 
$r_m={T}/{m}.$
\end{lemma}      

\begin{proof}
Firstly, if $0=t_0< t \leq t_1$, by (\ref{eq:4.5}), 
(\ref{eq:4.7}) and (\ref{eq:4.9}),
\begin{equation}
\label{eq:4.59}
|Y^n_{t}-{Y}^{n,m}_t| \leq \int_{t_0}^t |f(Y^{n}_{s}, B^n_{s})-f^n({Y}^{n,m}_{t_{0}}, 
B^n_{t_{0}})|ds,
\end{equation}
and
\begin{align}
\label{eq:4.60}
&|f(Y^{n}_{s}, B^n_{s}) - f^n({Y}^{n,m}_{t_{0}}, B^n_{t_{0}})|
\nonumber \\
\leq& |f(Y^{n}_{s}, B^n_{s})-f({Y}^{n,m}_s, B^n_{s})|+ |f({Y}^{n,m}_s, 
B^n_{s})-f({Y}^{n,m}_{t_{0}}, B^n_{s})|\nonumber\\
&  + |f({Y}^{n,m}_{t_{0}}, B^n_{s}) - f({Y}^{n,m}_{t_{0}}, B^n_{t_{0}})|+ 
|f({Y}^{n,m}_{t_{0}}, B^n_{t_0}) - f^n({Y}^{n,m}_{t_{0}}, B^n_{t_{0}})|.
\end{align}

By  part 2 of Lemma \ref{l:4.6}, 
\begin{equation}
\label{eq:4.61}
|f(Y^{n}_{s}, B^n_{s})-f({Y}^{n,m}_s, B^n_{s})|\leq Z_1|Y^n_{s}-{Y}^{n,m}_s|.
\end{equation}
Analogously, and by (\ref{eq:4.9}) and (\ref{eq:4.50}), for $s\leq t_1$,
\begin{align}
\label{eq:4.62}
|f({Y}^{n,m}_s, B^n_{s})-&f({Y}^{n,m}_{t_{0}}, B^n_{s})|
\leq Z_1|{Y}^{n,m}_{s}-{Y}^{n,m}_{t_{0}}|\notag\\
&\leq Z_1(s-t_0)|f^n({Y}^{n,m}_{t_{0}}, B^n_{t_0})|
\leq Z_1M_1\exp(M_2\left\|B^n\right\|_{\infty})r_m.	
\end{align}
By Lemmas \ref{l:4.2} and \ref{l:4.6},  for $t_0< s < t_1$, 
\begin{eqnarray}
\label{eq:4.63}
|f({Y}^{n,m}_{t_{0}}, B^n_{s}) - f({Y}^{n,m}_{t_{0}}, B^n_{t_{0}})|
& \leq&   Z_2|B_s^n- B^n_{t_0}| \leq  Z_2Kn^{1+\beta}(s-t_0) \nonumber \\
&\leq  &  Z_2Kn^{1+\beta}r_m.
\end{eqnarray}
Moreover, 
\begin{equation}
\label{eq:4.64}
 |f({Y}^{n,m}_{t_{0}}, B^n_{t_0}) - f^n({Y}^{n,m}_{t_{0}}, B^n_{t_{0}})|\leq I_7 + I_{8},
\end{equation}
where
\begin{eqnarray*}
I_7	&=& |b(h({Y}^{n,m}_{t_{0}}, B^n_{t_0}))|\left|\exp\left(-\int_0^{B^n_{t_0}} \sigma'(h({Y}^{n,m}_{t_{0}},u))du\right)\right.\nonumber \\
&&\hspace{3.6cm} \left. - \exp\left(-\int_0^{B^n_{t_0}} \sigma'(h^n(
{Y}^{n,m}_{t_{0}},u))du\right)\right|,\nonumber 
\end{eqnarray*}
$$
I_{8}	= \left|\exp\left(-\int_0^{B^n_{t_0}} \sigma'
(h^n({Y}^{n,m}_{t_{0}},u))du\right)\right||b(h({Y}^{n,m}_{t_{0}}, B^n_{t_0}))-b(h^n({Y}^{n,m}_{t_{0}}, B^n_{t_0}))|.
$$

Applying the mean value theorem for the exponential function and $\sigma'$, by  Lemmas \ref{l:4.7} and \ref{l:4.10}, (\ref{eq:4.17}), and boundedness of $b, \sigma'$,  then, for $n>N$,
\begin{eqnarray}
\label{eq:4.65}
I_7	&\leq& M_1M_3\exp(M_2\|B^n\|_{\infty})\int_0^{|B^n_{t_0}|}\left|h({Y}^{n,m}_{t_{0}},u)-h^n({Y}^{n,m}_{t_{0}},u)\right|du\nonumber\\
&\leq& M_1M_3\exp(M_2\|B^n\|_{\infty})\int_0^{|B^n_{t_0}|}\bar{M}^2\frac{N}{n}\exp{(\bar{M}N)}du\nonumber\\
&\leq& M_1M_3\bar{M}^2\exp(M_2\|B^n\|_{\infty})\exp{(\bar{M}N)}\|B^n\|_{\infty}\frac{N}{n}.
\end{eqnarray}

By Corollary  \ref{c:4.11}, (\ref{eq:4.17}) and  $b$ being Lipschitz, for $n>N$,
\begin{eqnarray}
\label{eq:4.66}
I_{8}	&\leq& M_4\exp(M_2\|B^n\|_{\infty})\left|h({Y}^{n,m}_{t_{0}},B^n_{t_0})-h^n({Y}^{n,m}_{t_{0}},B^n_{t_0})\right|\nonumber\\
&\leq& M_4\bar{M}^2\exp(M_2\|B^n\|_{\infty})\exp{(\bar{M}N)}\frac{N}{n}.
\end{eqnarray}

From (\ref{eq:4.64}), (\ref{eq:4.65}) and (\ref{eq:4.66}), 
\begin{eqnarray}
\label{eq:4.67}
&&|f({Y}^{n,m}_{t_{0}}, B^n_{t_0}) - f^n({Y}^{n,m}_{t_{0}}, B^n_{t_{0}})|\nonumber\\
&\leq& \bar{M}^2\exp(M_2\|B^n\|_{\infty})\exp{(\bar{M}N)}(M_1M_3\|B^n\|_{\infty}+M_4)\frac{N}{n}.
\end{eqnarray}

By (\ref{eq:4.60}), (\ref{eq:4.61}), (\ref{eq:4.62}), 
(\ref{eq:4.63}), (\ref{eq:4.67}) and (\ref{eq:4.58}), 
for $s\in[0,T]$,
\begin{equation}
\label{eq:4.68}
 |f(Y^{n}_{s}, B^n_{s}) - f^n({Y}^{n,m}_{t_{0}}, B^n_{t_{0}})|\leq Z_1|Y^n_{s}-
{Y}^{n,m}_s| + J_{n,m}.
\end{equation}

From (\ref{eq:4.59}), (\ref{eq:4.68}), for $0<t\leq t_1$,
$$|Y^n_{t}-{Y}^{n,m}_t|\leq \int_{t_0}^t Z_1|Y^n_{s}-{Y}^{n,m}_s|ds + 
\int_{t_0}^t J_{n,m}ds, $$
 and by Gronwall's lemma, for $t_0< t \leq t_1$ and $r_m={T}/{m}$,
\begin{equation}
\label{eq:4.69}
|Y^n_{t}-{Y}^{n,m}_t| \leq J_{n,m}r_m\exp\left(Z_1(t_1-t_0)\right).
\end{equation}

If $t_1< t \leq t_2$,
\begin{eqnarray}
\label{eq:4.70}
|Y^n_{t}-{Y}^{n,m}_t|&=& \left|Y_{t_1}^n + \int_{t_1}^t f(Y^{n}_{s}, B^n_{s})ds - {Y}^{n,m}_{t_{1}} - \int_{t_1}^t f^n({Y}^{n,m}_{t_{1}}, B^n_{t_{1}})ds\right|\nonumber\\
& \leq& |Y_{t_1}^n -  {Y}^{n,m}_{t_{1}}| + \int_{t_1}^t |f(Y^{n}_{s}, B^n_{s})-f^n({Y}^{n,m}_{t_{1}}, B^n_{t_{1}})|ds.
\end{eqnarray}
Proceeding similarly, taking $t_1$ instead of $t_0$, from 
(\ref{eq:4.68}),
\begin{equation}
\label{eq:4.71}
 |f(Y^{n}_{s}, B^n_{s}) - f^n({Y}^{n,m}_{t_{1}}, B^n_{t_{1}})|\leq Z_1|Y^n_{s}-
{Y}^{n,m}_s| + J_{n,m}.
\end{equation}

From (\ref{eq:4.69}), (\ref{eq:4.70}) and (\ref{eq:4.71}), for $t_1\leq t \leq t_2$,
\begin{eqnarray*}
|Y^n_{t}-{Y}^{n,m}_t|&\leq& J_{n,m}r_m\exp\left(Z_1(t_1-t_0)\right)+ \int_{t_1}^t Z_1|Y^n_{s}-{Y}^{n,m}_s|ds + \int_{t_1}^t J_{n,m} ds\nonumber\\
&\leq& J_{n,m}r_m\left(\exp\left(Z_1(t_1-t_0)\right)+1\right)+ \int_{t_1}^t Z_1|Y^n_{s}-{Y}^{n,m}_s|ds,
\end{eqnarray*}
and by  Gronwall's lemma,
\begin{eqnarray*}
|Y^n_{t}-{Y}^{n,m}_t|&\leq& J_{n,m}r_m\left(\exp\left(Z_1(t_1-t_0)\right)+1\right)\exp\left( Z_1(t_2-t_1)\right)\nonumber\\
&\leq& J_{n,m}r_m\left(\exp\left(Z_1(t_2-t_0)\right)+\exp\left( Z_1(t_2-t_1)\right)\right)\nonumber\\
&\leq& J_{n,m}r_m2\left(\exp\left(Z_1T\right)\right).
\end{eqnarray*}

Analogously, for $t_{k}<t\leq t_{k+1}$, $k=0, \cdots, m-1,$
\begin{align*}
& |Y^n_{t}-{Y}^{n,m}_t|	\nonumber\\
\leq&J_{n,m}r_m[\exp\left(Z_1(t_{k+1}-t_0)\right)+\exp\left(Z_1(t_{k+1}-t_1)\right)+\cdots +\exp\left(Z_1(t_{k+1}-t_k) \right)]\nonumber\\
\leq &J_{n,m}r_m(k+1)\left(\exp\left(Z_1T\right)\right)\leq 
J_{n,m}r_mm\left(\exp\left(Z_1T\right)\right)\nonumber\\
= & J_{n,m}T\left(\exp\left(Z_1T\right)\right),  
\end{align*}
which finishes the proof. 
\end{proof}

\begin{proposition}
\label{p:4.13}
Let $Y^n$ and ${Y}^{n, m}$ be given by (\ref{eq:4.5}) and (\ref{eq:4.9}), respectively. Then
$$P\Bigl(\limsup_{n\to \infty}\Bigl\{\|Y^n-{Y}^{n, n^2}\|_{\infty}>\alpha_n\Bigr\}\Bigr)=0,$$
where $\alpha_n$ is given by (\ref{eq:4.33}). 
\end{proposition}

\begin{proof}
By Lemma \ref{l:4.10}, there is $N=N(\omega) $ given by 
(\ref{eq:4.56}) such as for  $n, m,$ 
$$({Y}^{n,m}_t, B^n_t)\in [-N, N]\times [-N, N] \quad a.s.$$

The random variables $Z_1$ and $ Z_2$ given by 
(\ref{eq:4.25}) and (\ref{eq:4.26}) are bounded uniformly in $n$, and defining $M=\max\left\{M_1, M_2, M_3, M_4, M_5\right\}$,
$$Z_1  \leq\exp(2MN)[ M^2N+M],\quad Z_2\leq 2M^2\exp(MN),$$
and by (\ref{eq:4.58}),
\begin{align}
\label{eq:4.72}
J_{n,m}\leq &\exp(3MN)(M^3N+M^2)r_m + 2M^2K\exp(MN)n^{1+\beta}r_m\notag \\
&+M^2\exp(2MN)(M^2N+M)\frac{N}{n}.	
\end{align}

From  (\ref{eq:4.57}) and (\ref{eq:4.72}), for $n> N$ and taking $m=n^2$, then $r_m={T}/{n^2}$ and
 \begin{align*}
|Y^n_{t}-{Y}^{n,n^2}_t|&\leq  J_{n,m}T\left(\exp\left(Z_1T\right)\right)\nonumber\\
&\leq \Bigl\{ \exp(3MN)(M^3N+M^2)r_m + 2M^2K\exp(MN)n^{1+\beta}r_m\Bigr.\nonumber\\ 
&\Bigl. \ \  +M^2\exp(2MN)(M^2N+M)\frac{N}{n} \Bigr\} T\exp\left(T\exp(2MN)[M^2N+M]\right)\nonumber\\
&\leq Z_3\frac{1}{n^{1-\beta}},
\end{align*}
where the random variable $Z_3$ is defined by
\begin{eqnarray*}
Z_3&=&\left\{\exp(3MN)(M^3N+M^2)T + 2M^2K\exp(MN)T \right.\\
	 &&\left. \ +M^2N\exp(2MN)(M^2N+M)\right\}T\exp\left(T\exp(2MN)[M^2N+M]\right).
\end{eqnarray*}

Therefore we obtain
\begin{align*}
 &P\Bigl(\limsup_{n\to \infty}\bigl\{\|Y^n-{Y}^{n, n^2}\|_{\infty}>\alpha_n\bigr\}\Bigr)\\
\leq&  P\Bigl(\limsup_{n\to \infty}\bigl\{Z_3 >n^{1-\beta}\alpha_n\bigr\}\Bigr)
= P\Bigl(\limsup_{n\to \infty}\bigl\{Z_3 >n^{1/2+\delta}(\log n)^{5/2}\bigr\}\Bigr)=0. 
\end{align*}
\end{proof}

\begin{corollary}
\label{c:4.14}
Let $Y$ and ${Y}^{n, m}$  be given by (\ref{eq:2.11}) and (\ref{eq:4.9}), 
respectively. Then
$$P\Bigl(\limsup_{n\to \infty}\{\|Y-{Y}^{n, n^2}\|_{\infty}>\alpha_n\}\Bigr)=0,$$
where $\alpha_n$ is given by (\ref{eq:4.33}).
\end{corollary}

\begin{proof}
 By the proofs of Propositions \ref{p:4.8} and \ref{p:4.13}, and replacing $Z$ and $Z_3$ by $2Z$ and $2Z_3$, respectively,
\begin{align*}
&P\Bigl(\limsup_{n\to \infty}\{\|Y-{Y}^{n, n^2}\|_{\infty}>\alpha_n\}\Bigr)\\
\leq& P\Bigl(\limsup_{n\to \infty}\{\|Y-Y^{n}\|_{\infty}>\alpha_n/2\}\Bigr)
+P\Bigl(\limsup_{n\to \infty}\{\|Y^n-{Y}^{n, n^2}\|_{\infty} >\alpha_n/2\}\Bigr)=0. 
\end{align*}
\end{proof}

We define for each $n=1,2, \dots$,
\begin{equation}
\label{eq:4.73}
\tilde{X}^n_t=h(Y_t^n, B_t^n), 
\end{equation}
where $B^n$ and $Y^n$ are given by (\ref{eq:2.6})-(\ref{eq:2.7}) and (\ref{eq:4.5}), respectively.
\begin{proposition}
\label{p:4.15}
For any $\tilde{C}>0,$
$$P\Bigl(\limsup_{n\to \infty}\{\|X-\tilde{X}^n\|_{\infty}>\tilde{C}\alpha_n\}\Bigr)=0,$$
where $X$ is given by (\ref{eq:2.9}) and $\alpha_n$ is given by 
(\ref{eq:4.33}).
\end{proposition}

\begin{proof}
  For convenience of notation we put $\tilde{C}=1$. Then
\begin{equation}
\label{eq:4.74}
|X_t-\tilde{X}^n_t|=|h(Y_t, B^H_t)-h(Y_t^n, B_t^n)|\leq I_9(t) + I_{10}(t),
\end{equation}
where 
$$I_9(t)=|h(Y_t, B^H_t)-h(Y_t^n, B^H_t)|$$
$$I_{10}(t)=|h(Y_t^n, B^H_t)-h(Y_t^n, B_t^n)|.$$

By (\ref{eq:4.19}),
\begin{equation}
\label{eq:4.75}
I_9(t)	\leq Z_4\|Y-Y^n\|_{\infty},
\end{equation}
where $Z_4=\exp(M_2\left\|B^H\right\|_{\infty})$,

Proceeding similary as in the proof of (\ref{eq:4.48}), then from  
Proposition \ref{p:4.8},
\begin{eqnarray*}
&&P\Bigl(\limsup_{n\to \infty}\{Z_4\left\|Y-Y^n\right\|_{\infty}>\alpha_n/2\}\Bigr)\\
 & \leq  & P\Bigl(\limsup_{n\to \infty}\{\|Y-Y^n\|_{\infty}>
n^{-1/2+\beta+\delta/2}(\log n)^{5/2}\}\Bigr)=0.
\end{eqnarray*}
Thus, by (\ref{eq:4.75}),
\begin{equation}
\label{eq:4.76}	
P\Bigl(\limsup_{n\to \infty}\{\|I_9\|_{\infty}>\alpha_n/2\}\Bigr)=0.
\end{equation}

From (\ref{eq:2.10}), boundedness of $\sigma$, and the mean value theorem, 
$$I_{10}(t)\leq M_5\|B^H-B^n\|_{\infty},$$
and by  (\ref{eq:2.8}),
\begin{equation}
\label{eq:4.77}	
P\Bigl(\limsup_{n\to \infty}\{\|I_{10}\|_{\infty}>\alpha_n/2\}\Bigr)=0.
\end{equation}
The result follows from (\ref{eq:4.74}), (\ref{eq:4.76}) and 
(\ref{eq:4.77}). 
\end{proof}

For the final step of the proof of the theorem we go to (\ref{eq:4.11}).
	
By (\ref{eq:2.9}), (\ref{eq:4.73}), (\ref{eq:4.12}) and
Proposition \ref{p:4.15},
\begin{equation}
\label{eq:4.78}
P\Bigl(\limsup_{n\to \infty}\{\|H_{1}\|_{\infty}>\alpha_n/4\}\Bigr)=0.
\end{equation}

By (\ref{eq:4.52}) there is $N=N(\omega)$ given by 
(\ref{eq:4.56}) such that for each $n,m$,
\begin{equation}
\label{eq:4.79}
({Y}^{n,m}_t, B^n_t)\in [-N, N]\times [-N, N] \quad{\rm a.s.}
\end{equation}
From (\ref{eq:4.13}) and (\ref{eq:4.19}),
\begin{equation}
\label{eq:4.80}
H_{2}(t)\leq \exp(M_2\|B^n\|_{\infty})|Y_t^n-{Y}_t^{n, n^2}|
\leq Z_5|Y_t^n-{Y}_t^{n, n^2}|,
\end{equation}
where $Z_5=\exp(M_2N)$. 
By Proposition \ref{p:4.13} and (\ref{eq:4.80}), 
\begin{eqnarray}
\label{eq:4.81}
&&P\Bigl(\limsup_{n\to \infty}\left\{\left\|H_{2}\right\|_{\infty}>\alpha_n/4\right\}\Bigr)\leq 
P\Bigl(\limsup_{n\to \infty}\{4Z_5\|Y^n-{Y}^{n, n^2}\|_{\infty}>\alpha_n\}\Bigr) \nonumber\\ 
&\leq & P\Bigl(\limsup_{n\to \infty}\{\|Y^n-{Y}^{n, n^2}\|_{\infty}>
n^{-1/2+\beta+\delta/2}(\log n)^{5/2}\}\Bigr)=0.
\end{eqnarray}

On account of (\ref{eq:4.79}) and (\ref{eq:4.14}), for all $n>N$, applying  Lemma 
\ref{l:4.7}, 
$$
H_{3}(t)\leq  Z_6\frac{1}{n},$$
where  $Z_6=\bar{M}^2N\exp(\bar{M}N)$.
Then,
\begin{align}
\label{eq:4.82}
P\Bigl(\limsup_{n\to \infty}\left\{\left\|H_{3}\right\|_{\infty}>\alpha_n\right\}\Bigr)&\leq 
P\Bigl(\limsup_{n\to \infty}\Bigl\{Z_6>n\alpha_n\Bigr\}\Bigr)\nonumber\\
&\leq P\Bigl(\limsup_{n\to \infty}\{Z_6>n^{1/2+\beta+\delta}(\log n)^{5/2}\}\Bigr)=0.
\end{align}

 The final result follows from (\ref{eq:4.11}), (\ref{eq:4.12}), 
(\ref{eq:4.13}), (\ref{eq:4.14}), (\ref{eq:4.78}), (\ref{eq:4.81})  and 
(\ref{eq:4.82}).
 
\bibliographystyle{amsplain}

\end{document}